\documentclass[12pt,a4paper]{amsart}
\usepackage{latexsym}
\usepackage{amssymb}
\usepackage{calc}
\usepackage[all]{xy}
\UseTips
\newdir^{c}{{}*!/-3pt/\dir^{(}}
\newdir_{c}{{}*!/-3pt/\dir_{(}}
\usepackage{aliascnt} 
\usepackage[colorlinks=true,backref,hyperindex]{hyperref}
\usepackage[dvipsnames,usenames]{color}
\hypersetup{linkcolor=RawSienna,anchorcolor=BurntOrange,citecolor=OliveGreen,filecolor=BlueViolet,menucolor=Yellow,urlcolor=OliveGreen}

\usepackage[marginratio={1:1, 2:3}]{geometry}

\numberwithin{equation}{section}
\theoremstyle{plain}   

\newtheorem{thm}{Theorem}[section]    
\newtheorem*{thm*}{Theorem}           

\newcommand{\maketheorem}[2]{%
    \newaliascnt{#1}{thm}
    \newtheorem{#1}[#1]{#2}
    \aliascntresetthe{#1}
    \expandafter\def\csname #1autorefname\endcsname{#2}
    \newtheorem{#1*}{#2}
}%

\maketheorem{prop}{Proposition}
\maketheorem{cor}{Corollary}
\maketheorem{lem}{Lemma}
\maketheorem{conj}{Conjecture}

\theoremstyle{definition}    

\maketheorem{defn}{Definition}

\theoremstyle{remark}    

\maketheorem{rem}{Remark}
\maketheorem{ex}{Example}

\makeatletter
\renewcommand{\theenumi}{\@alph\c@enumi}
\renewcommand{\theenumii}{\@roman\c@enumii}

\renewcommand{\theenumiii}{\@Alph\c@enumiii}

\renewcommand{\theenumiv}{\@Roman\c@enumiv}

\renewcommand{\labelitemi}{$\m@th\circ$}
\renewcommand{\labelitemii}{$\m@th\diamond$}
\renewcommand{\labelitemiii}{$\m@th\star$}
\renewcommand{\labelitemiv}{$\m@th\cdot$}
\makeatother

\newcommand{\Fr}{\sigma}
\newcommand{\FrG}{\Fr_\gamma}
\newcommand{\CO}{\mathcal{O}}
\newcommand{\CM}{\mathcal{M}}
\renewcommand{\r}{\mathrm{r}}
\newcommand{\Hom}{\operatorname{Hom}}
\newcommand{\End}{\operatorname{End}}
\newcommand{\usc}{\underline{\phantom{M}}}
\newcommand{\tensor}{\otimes}
\newcommand{\Spec}{\operatorname{Spec}}
\newcommand{\defeq}{\stackrel{\scriptscriptstyle \operatorname{def}}{=}}
\newcommand{\Gen}{\mathsf{Gen}}
\newcommand{\Ker}{\operatorname{ker}}
\newcommand{\Coker}{\operatorname{coker}}
\newcommand{\Image}{\operatorname{image}}
\renewcommand{\min}{\text{min}}
\renewcommand{\to}[1][]{\xrightarrow{\ #1\ }}
\newcommand{\onto}[1][]{\protect{\xrightarrow{\ #1\ }\hspace{-0.8em}\rightarrow}}
\newcommand{\into}[1][]{\lhook \joinrel \xrightarrow{\ #1\ }}
\renewcommand{\frm}{\mathfrak{m}}

\renewcommand{\phi}{\varphi}
\renewcommand{\min}{\mathrm{min}}

\newcommand{\CL}{\mathcal{L}}
\newcommand{\Ext}{\operatorname{Ext}}
\newcommand{\id}{\operatorname{id}}

\newcommand{\CohG}{\mathbf{Coh}_\gamma}

\newcommand{\CrysG}{\mathbf{Crys}_\gamma}
\newcommand{\MinG}{\mathbf{Min}_\gamma}

\newcounter{themargin}
\def\m?#1{\textcolor{Mahogany}{\textbf{???$^{\text{\arabic{themargin}}}$}}{\marginpar{\footnotesize\color{Mahogany}\fbox{\parbox{\marginparwidth}{\textbf{mnl --- \arabic{themargin} ---}\addtocounter{themargin}{1}\\ #1}}} \immediate\write16{}%
\immediate\write16{Warning: There was still a question mark . . . }%
\immediate\write16{}}}

\begin{document}

\title[Minimal $\gamma$--sheaves]{Minimal $\gamma$--sheaves}
\author{Manuel Blickle}
\address{Mathematik Essen\\ Universit\"at Duisburg-Essen\\ 45117 Essen\\ Germany}
\email{manuel.blickle@uni-due.de}
\urladdr{www.mabli.org}
\thanks{During the preparation of this article the author was supported by the \emph{DFG Schwerpunkt Komplexe Geometrie}.}

\begin{abstract}
    In this note we show that finitely generated unit $\CO_X[\Fr]$--modules for $X$ regular and $F$--finite have a minimal root (in the sense of \cite{Lyub} Definition~3.6). This question was asked by Lyubeznik and answered by himself in the complete case.
    In fact, we construct a minimal subcategory of the category of coherent $\gamma$--sheaves (in the sense of \cite{BliBoe.CartierCrys}) which is equivalent to the category of $\gamma$--crystals. Some applications to tight closure are included at the end of the paper.
\end{abstract}
\maketitle

\section{Introduction}
In \cite{Lyub} introduces the category of finitely generated unit $R[\Fr]$--modules and applies the resulting theory successfully to study finiteness properties of local cohomology modules. One of the main tools in proving results about unit $R[\Fr]$--modules is the concept of a generator or root. In short, a generator (later on called $\gamma$--sheaf) is a finitely generated module $M$ together with a map $\gamma: M \to \Fr^*M$. By repeated application of $\Fr^*$ to this map one obtains a direct limit system, whose limit we call $\Gen M$. One checks easily that $\gamma$ induces an map $\Gen M \to \Fr^*\Gen M$ which is an isomorphism. A finitely generated unit $R[\Fr]$--module $\CM$ is precisely a module which is isomorphic to $\Gen M$ for some $\gamma$--sheaf $(M,\gamma)$, it hence comes equipped with an isomorphism $\CM \cong \Fr^* \CM$. Of course, different $\gamma$--sheaves may generate isomorphic unit $R[\Fr]$--modules so the question arises if there is a unique minimal (in an appropriate sense) $\gamma$--sheaf that generates a given unit $R[\Fr]$--module. In the case that $R$ is complete, this is shown to be the case in \cite{Lyub} Theorem 3.5. In \cite{Bli.int} this is extended to the case that $R$ is local (at least if $R$ is $F$-finite). The purpose of this note is to prove this in general, i.e for any $F$--finite regular ring $R$ (see \autoref{t.MinFunct}). A notable point in the proof is that it does not rely on the hard finiteness result \cite{Lyub} Theorem 4.2, but only on the (easier) local case of it which is in some sense proven here \emph{en passant} (see \autoref{r.finitelength}).

The approach in this note is not the most direct one imaginable since we essentially develop a theory of minimal $\gamma$--sheaves from scratch (\autoref{s.nilpGamma}). However, with this theory at hand, the results on minimal generators are merely a corollary. The ideas in this paper have two sources. Firstly, the ongoing project \cite{BliBoe.CartierCrys} of the author with Gebhard B\"ockle lead to a systematic study of $\gamma$--sheaves (the notation $\gamma$--sheaf is chosen to remind of the notion of a \textbf{g}enerator introduced in \cite{Lyub}). Secondly, insight gained from the $D$--module theoretic viewpoint on generalized test ideals developed in \cite{BliMusSmi.hyp} lead to the observation that these techniques can be successfully applied to study $\gamma$--sheaves.

In the final \autoref{s.final} we give some applications of the result on the existence of minimal $\gamma$--sheaves. First, we show that the category of minimal $\gamma$--sheaves is equivalent to the category $\gamma$--crystals of \cite{BliBoe.CartierCrys}. We show that a notion from tight closure theory, namely the parameter test module, is a global object (\autoref{t.GlobalParTest}). Statements of this type are notoriously hard in the theory of tight closure. Furthermore, we give a concrete description of minimal $\gamma$--sheaves in a very simple case (\autoref{t.MinimalTestIdealChar}), relating it to the generalized test ideals studied in \cite{BliMusSmi.hyp}. This viewpoint also recovers (and slightly generalizes, with new proofs) the main results of \cite{BliMusSmi.hyp} and \cite{AlvBliLyub.GeneratorsDmod}. A similar generalization, however using different (but related) methods, was recently obtained independently by Lyubeznik, Katzman and Zheng in \cite{LyubKazmZheng}.

\subsection*{Notation}
Throughout we fix a \emph{regular} scheme $X$ over a field $k \supseteq \mathbb{F}_q$ of characteristic $p > 0$ (with $q=p^e$ fixed). We further assume that $X$ is $F$--finite, i.e. the Frobenius morphism $\Fr: X \to X$, which is given by sending $f \in \CO_X$ to $f^q$, is a finite morphism\footnote{It should be possible to replace the assumption of $F$--finiteness to saying that if $X$ is a $k$--scheme with $k$ a filed that the relative Frobenius $\Fr_{X/k}$ is finite. This would extend the results given here to desirable situations such as $X$ of finite type over a field $k$ with $[k:k^q]=\infty$. The interested reader should have no trouble to adjust our treatment to this case.}. In particular, $\Fr$ is affine. This allows to reduce in many arguments below to the case that $X$ itself is affine and I will do so if convenient. We will use without further mention that because $X$ is regular, the Frobenius morphism $\Fr \colon X \to X$ is flat such that $\Fr^*$ is an exact functor (see \cite{Kunz}).

\section{Minimal \texorpdfstring{$\gamma$}{Gamma}--sheaves}\label{s.nilpGamma}
We begin with recalling the notion of $\gamma$--sheaves and nilpotence.
\begin{defn}\label{d.gammaSheaf}
    A \emph{$\gamma$--sheaf} on $X$ is a pair $(M,\gamma_M)$ consisting of a quasi-coherent $\CO_X$--module $M$ and a $\CO_X$--linear map $\gamma: M \to \Fr^*M$. A $\gamma$ sheaf is called \emph{coherent} if its underlying sheaf of $\CO_X$--modules is coherent.

    A $\gamma$--sheaf $(M,\gamma)$ is called \emph{nilpotent} (of order $n$) if $\gamma^n \defeq \Fr^{n*}\gamma \circ \Fr^{(n-1)*}\gamma \circ \ldots \circ \Fr^*\gamma \circ \gamma = 0$ for some $n > 0$. A $\gamma$--sheaf is called \emph{locally nilpotent} if it is the union of nilpotent $\gamma$ subsheaves.
\end{defn}
Maps of $\gamma$--sheaves are maps of the underlying $\CO_X$--modules such that the obvious diagram commutes.
The following proposition summarizes some properties of $\gamma$--sheaves, for proofs and more details see \cite{BliBoe.CartierCrys}.
\begin{prop}\label{prop.nilp-abcat}
  \begin{enumerate}
    \item The set of $\gamma$--sheaves forms an abelian category which is closed under extensions.
    \item The coherent, nilpotent and locally nilpotent $\gamma$-sheaves are abelian subcategories, also closed under extension.
  \end{enumerate}
\end{prop}
\begin{proof}
    The point in the first statement is that the $O_X$--module kernel, co-kernel and extension of (maps of) $\gamma$--sheaves naturally carrys the structure of a $\gamma$--sheaf. This is really easy to verify such that we only give the construction of the $\gamma$--structure on the kernel as an illustration.
    Recall that we assume that $X$ is regular such that $\Fr$ is flat, hence $\Fr^*$ is an exact functor. If $\phi:  M \to N$ is a homomorphism of $\gamma$-sheaves, i.e. a commutative diagram
    \[
        \xymatrix{ M \ar[rr]^{\phi} \ar[d]_{\gamma_{M}} && N \ar[d]^{\gamma_{N}} \\
        {\Fr^* M} \ar[rr]^{\Fr^*\phi} && {\Fr^*N}
        }
    \]
    which induces a map $\Ker \phi \to \Ker (\Fr^*\phi)$. Since $\Fr^*$ is exact, the natural map $\Fr^*(\Ker \phi) \to \Ker (\Fr^*\phi)$ is an isomorphism. Hence the composition
    \[
    \ker \phi \to \ker(\Fr^*\phi) \to[\cong] \Fr^*(\ker \phi)
    \]
    equips $\ker \phi$ with a natural structure of a $\gamma$--sheaf.

    The second part of \autoref{prop.nilp-abcat} is also easy to verify such that we leave it to the reader, cf. the proof of \autoref{lem.nil-iso} below.
\end{proof}

\begin{lem}\label{lem.nil-iso}
    A morphism $\phi: M \to N$ of $\gamma$-sheaves is called \emph{nil-injective} (resp. \emph{nil-surjective}, \emph{nil-isomorphism}) if its kernel (resp. cokernel, both) is locally nilpotent.
    \begin{enumerate}
        \item If $N$ is coherent and $\phi$ is \emph{nil-injective} (resp. \emph{nil-surjective}) then $\Ker \phi$ (resp. $\Coker \phi$) is nilpotent.
        \item Kernel and cokernel of $\phi$ are nilpotent (of order $n$ and $m$ resp.) if and only if there is, for some $k \geq 0$ ($k=n+m$), a map $\psi: N \to \Fr^{k*}M$ such that $\gamma^k_M = \psi \circ \phi$.
        \item If $N$ is nilpotent of degree $\leq n$ (i.e. $\gamma_N^n=0$) and $N' \subseteq N$ contains the kernel of $\gamma^i_N$ for $1 \leq i \leq n$, then $N'$ is nilpotent of degree $\leq i$ and $N/N'$ is nilpotent of degree $\leq n-i$.
    \end{enumerate}
\end{lem}
\begin{proof}
    The first statement is clear since $X$ is noetherian. For the second statement consider the diagram obtained from the exact sequence $0 \to K \to M \to N \to C \to 0$.
    \[
    \xymatrix{ 0 \ar[r] &K \ar[r]\ar[d] &M \ar[r]\ar[d] &N \ar[r]\ar[d]\ar@{..>}[ldd]\ar[ld]_{\psi} &C \ar[r]\ar[d]^0 &0 \\
        0 \ar[r] &\Fr^{n*}K \ar[r]\ar[d]^0 &\Fr^{n*}M \ar[r]\ar[d] &\Fr^{n*}N \ar[r]\ar[d] &\Fr^{n*}C \ar[r]\ar[d] &0 \\
        0 \ar[r] &\Fr^{(n+m)*}K \ar[r] &\Fr^{(n+m)*}M \ar[r] &\Fr^{(n+m)*}N \ar[r] &\Fr^{(n+m)*}C \ar[r] &0
    }
    \]
    If there is $\psi$ as indicated, then clearly the leftmost and rightmost vertical arrows of the first row are zero, i.e. $K$ and $C$ are nilpotent. Conversely, let $K=\Ker \phi$ be nilpotent of degree $n$ and $C=\Coker \phi$ be nilpotent of degree $m$. Then the top right vertical arrow and the bottom left vertical arrow are zero. This easily implies that there is a dotted arrow as indicated, which will be the sought after $\psi$.

    For the last part, the statement about the nilpotency of $N'$ is trivial. Consider the short exact sequence $0 \to N' \to N \to N/N' \to 0$ and the diagram one obtains  by considering $\Fr^{(n-i)*}$ and $\Fr^{n*}$ of this sequence.
    \[
    \xymatrix{
    0 \ar[r] & N'\ar[r]\ar[d] & N\ar[r]\ar[d]^{\gamma^{n-i}} & N/N'\ar[r]\ar[d] & 0 \\
    0 \ar[r] & \Fr^{(n-i)*}N'\ar[r]\ar[d]^{0} & \Fr^{(n-i)*}N\ar[r]\ar[d]^{\Fr^{(n-i)*}\gamma^i} & \Fr^{(n-i)*}(N/N')\ar[r]\ar[d] & 0 \\
    0 \ar[r] & \Fr^{n*}N'\ar[r] & \Fr^{n*}N\ar[r] & \Fr^{n*}(N/N')\ar[r] & 0
    }
    \]
    The composition of the middle vertical map is $\gamma_N^n$ which is zero by assumption. To conclude that the top right vertical arrow is zero one uses the fact that $\sigma^{(n-i)^*}N' \supseteq \sigma^{(n-i)*}\Ker \gamma^i = \Ker (\sigma^{(n-i)^*}\gamma^i)$. With this it is an easy diagram chase to conclude that the top right vertical map is zero.
\end{proof}

\begin{lem}\label{t.lemma.basic}
Let $M \to[\phi] N$ be a map of $\gamma$--sheaves. Let $N' \subseteq N$ be such that $N/N'$ is nilpotent (hence $N' \subseteq N$ is a nil-isomorphism). Then $M/(\phi^{-1}N')$ is also nilpotent.
\end{lem}
\begin{proof}
    If $\phi$ is injective/surjective, the Snake Lemma shows that $M/(\phi^{-1}N')$ injects/surjects to $N/N'$. Now split $\phi$ into $M \onto \Image \phi \into N$.
\end{proof}

If $(M,\gamma)$ is a $\gamma$--sheaf, then $\Fr^*M$ is naturally a $\gamma$--sheaf with structural map $\Fr^*\gamma$. Furthermore, the map $\gamma: M \to \Fr^*M$ is then a map of $\gamma$--sheaves which is a \emph{nil-isomorphism}, i.e. kernel and cokernel are nilpotent. We can iterate this process to obtain a directed system
\begin{equation}\label{eq.Gen}
    M \to[\gamma] \Fr^*M \to[\Fr^*\gamma] \Fr^{2*}M \to[\Fr^{2*}\gamma] \ldots
\end{equation}
whose limit we denote by $\Gen M$. Clearly $\Gen M$ is a $\gamma$--sheaf whose structural map $\gamma_{\Gen M}$ is injective. In fact, it is an isomorphism since clearly $\Fr^* \Gen M \cong \Gen M$. Note that even if $M$ is coherent, $\Gen M$ is generally not coherent. Furthermore, let $\overline{M}$ be the image of $M$ under the natural map $M \to \Gen M$. Then, if $M$ is coherent, so is $\overline{M}$ and the map $M \onto \overline{M}$ is a nil-isomorphism. Since $\overline{M}$ is a $\gamma$--submodule of $\Gen M$ whose structural map is injective, the structural map $\overline{\gamma}$ of $\overline{M}$ is injective as well.
\begin{prop}
  The operation that assigns to each $\gamma$--sheaf $M$ its image $\overline{M}$ in $\Gen M$ is an end-exact functor (preserves exactness only at the end of sequences) from $\CohG(X)$ to $\CohG(X)$. The kernel $M^\circ=\bigcup \Ker \gamma^i_M$ of the natural map $M \to \overline{M}$ is the maximal (locally) nilpotent subsheaf of $M$.
\end{prop}
\begin{proof}
    The point is that one has a functorial map between the exact functors $\operatorname{id} \to \Gen$. An easy diagram chase shows that the image of such a functorial map is an end-exact functor (see for example \cite[2.17 Appendix 1]{Katz.RigidLocal}).
    The verification of the statement about $M^\circ$ left to the reader.
\end{proof}
Such $\gamma$--submodules with injective structural map enjoy a certain minimality property with respect to nilpotent subsheaves:
\begin{lem}\label{lem.inj-char}
  Let $(M,\gamma)$ be a $\gamma$--sheaf. The structural map $\gamma_M$ is injective if and only if $M$ does not have a non-trivial nilpotent subsheaf.
\end{lem}
\begin{proof}
  Assume that the structural map of $M$ is injective. This implies that the structural map of any $\gamma$-subsheaf of $M$ is injective. But a $\gamma$--sheaf with injective structural map is nilpotent if and only it is zero.

  Conversely, $\Ker \gamma_M$ is a nil-potent subsheaf of $M$. If $\gamma_M$ is not injective it is nontrivial.
\end{proof}

\subsection{Definition of minimal \texorpdfstring{$\gamma$}{Gamma}--sheaves}

\begin{defn}\label{d.minimal}
    A coherent $\gamma$--sheaf $M$ is called \emph{minimal} if the following two conditions hold.
    \begin{enumerate}
      \item $M$ does not have nontrivial nilpotent subsheaves.
      \item $M$ does not have nontrivial nilpotent quotients.
    \end{enumerate}
\end{defn}
A simple consequence of the definition is
\begin{lem}\label{lem.inj-sub.sur-quot}
  Let $M$ be a $\gamma$-sheaf. If $M$ satisfies (a) then any $\gamma$--subsheaf of $M$ also satisfies (a). If $M$ satisfies (b) the so does any quotient.
\end{lem}
\begin{proof} Immediate from the definition.
\end{proof}
As the preceding \autoref{lem.inj-char} shows, (a) is equivalent to the condition that the structural map $\gamma_M$ is injective. We give a concrete description of the second condition.
\begin{prop}\label{prop.char-quot}
  For a coherent $\gamma$--sheaf $M$, the following conditions are equivalent.
  \begin{enumerate}
    \item $M$ does not have nontrivial nilpotent quotients.
    \item For any map of $\gamma$-sheaves $\phi: N \to M$, if $\gamma_M(M) \subseteq \phi(\Fr^*N)$ (as subsets of $\Fr^*M$) then $\phi$ is surjective.
  \end{enumerate}
\end{prop}
\begin{proof}
  I begin with showing the easy direction that (a) implies (b): Note that the condition $\gamma_M(M) \subseteq \phi(\Fr^*N)$ in (b) precisely says that the induced structural map on the cokernel of $N \to M$ is the zero map, thus in particular $M/\phi(N)$ is a nilpotent quotient of $M$. By assumption on $M$, $M/\phi(N)=0$ and hence $\phi(N)=M$.

  Let $M \onto C$ be such that $C$ is nilpotent. Let $N \subseteq M$ be its kernel. We have to show that $N=M$. The proof is by induction on the the order of nilpotency of $C$ (simultaneously for all $C$). If $C=M/N$ is nilpotent of order 1 this means precisely that $\gamma(M) \subseteq \Fr^*N$, hence by (b) we have $N=M$ as claimed. Now let $N$ be such that the nilpotency order of $C \defeq M/N$ is equal to $n \geq 2$. Consider the $\gamma$-submodule $N'= \pi^{-1}(\Ker \gamma_C)$ of $M$. This $N'$ clearly contains $N$ and we have that $M/N' \cong C/(\Ker \gamma_C)$. By the previous \autoref{lem.nil-iso} we conclude that the nilpotency order of $M/N'$ is $\leq n-1$. Thus by induction $N'=M$. Hence $M/N = N'/N \cong \Ker \gamma_C$ is of nilpotency order 1. Again by the base case of the induction we conclude that $M=N$.
\end{proof}
These observations immediately lead to the following corollary.
\begin{cor}\label{t.min=minroot}
  A coherent $\gamma$--sheaf $M$ is minimal if and only if the following two conditions hold.
  \begin{enumerate}
    \item The structural map of $M$ is injective.
    \item If $N \subseteq M$ is a subsheaf such that $\gamma(M) \subseteq \Fr^*N$ then $N = M$.
  \end{enumerate}
\end{cor}
The conditions in the Corollary are essentially the definition of a \emph{minimal root} of a finitely generated unit $R[\Fr]$--module in \cite{Lyub}. The finitely generated unit $R[\Fr]$--module generated by $(M,\gamma)$ is of course $\Gen M$. Lyubeznik shows in the case that $R$ is a complete regular ring, that minimal roots exist. In \cite[Theorem 2.10]{Bli.int} I showed how to reduce the local case to the complete case if $R$ is $F$--finite. For convenience we give a streamlined argument of the result in the local case in the language of $\gamma$--sheaves.

\subsection{Minimal \texorpdfstring{$\gamma$}{Gamma}--sheaves over local rings}
The difficult part in establishing the existence of a minimal root is to satisfy condition (b) of \autoref{d.minimal}. The point is to bound the order of nilpotency of any nilpotent quotient of a fixed $\gamma$--sheaf $M$.
\begin{prop}\label{t.MainLocalCase}
    Let $(R,\frm)$ be regular, local and $F$--finite. Let $M$ be a coherent $\gamma$--sheaf and $N_i$ be a collection of $\gamma$--sub-sheaves which is closed under finite intersections and such that $M/N_i$ is nilpotent for all $i$. Then $M/\bigcap N_i$ is nilpotent.
\end{prop}
\begin{proof}
    Since $R$ is regular, local and $F$--finite, $R$ is via $\Fr$ a free $R$--module of finite rank. Hence $\Fr^*$ is nothing but tensorisation with a free module of finite rank. Such an operation commutes with the formation of inverse limits such that $\Fr^* \bigcap N_i = \bigcap(\Fr^*N_i)$ and hence $\bigcap N_i$ is a $\gamma$--subsheaf of $M$. Clearly we may replace $M$ by $M/\bigcap N_i$ such that we have $\bigcap N_i= 0$. By faithfully flatness of completion $M$ is nilpotent if and only if $\hat{R} \tensor_R M$ is a nilpotent $\gamma$--sheaf over $\hat{R}$ (and similar for all $M/N_i$). Hence we may assume that $(R,\frm)$ is complete. We may further replace $M$ by its image $\overline{M}$ in $\Gen M$. Thus we may assume that $M$ has injective structural map $\gamma: M \subseteq \Fr^*M$. We have to show that $M=0$.

    By the Artin-Rees Lemma (applied to $M \subseteq \Fr^*M$) there exists $t \geq 0$ such that for all $s > t$
    \[
        M \cap \frm^s\Fr^*M \subseteq \frm^{s-t}(M \cap \frm^t\Fr^*M) \subseteq \frm^{s-t}M\; .
    \]
    By Chevalleys's Theorem in the version of \cite[Lemma 3.3]{Lyub}, for some $s \gg 0$ (in fact $s \geq t+1$ will suffice) we find $N_i$ with $N_i \subseteq \frm^s M$. Possibly increasing $s$ we may assume that $N_i \not\subseteq \frm^{s+1} M$ (unless, of course $N_i=0$ in which case $M/N_i = M$ is nilpotent $\Rightarrow$ $M=0$ since $\gamma_M$ is injective, and we are done). Combining these inclusions we get
    \[
    \begin{split}
        N_i \subseteq \Fr^*N_i \cap M &\subseteq \Fr^*(\frm^sM) \cap M \\
            &\subseteq (\frm^s)^{[q]}\Fr^*M \cap M \subseteq \frm^{sq}\Fr^*M \cap M \\
            &\subseteq \frm^{sq-t}M \; .
    \end{split}
    \]
    But since $sq-t \geq s+1$ for our choice of $s \geq t+1$ this is a contradiction (to the assumption $N_i \neq 0$) and the result follows.
\end{proof}
\begin{cor}\label{t.MainLocal}
    Let $R$ be regular, local and $F$--finite and $M$ a coherent $\gamma$--sheaf. Then $M$ has a nil-isomorphic subsheaf without non-zero nilpotent quotients (i.e. satisfying (b) of the definition of minimality). In particular, $M$ is nil-isomorphic to a minimal $\gamma$--sheaf.
\end{cor}
\begin{proof}
    Let $N_i$ be the collection of all nil-isomorphic subsheaves of $M$. This collection is closed under finite intersection: If $N$ and $N'$ are two such, then \autoref{t.lemma.basic} shows that $N\cap N'$ is a nil-isomorphic subsheaf of $N$. Since composition of nil-isomorphisms are nil-isomorphisms it follows that $N \cap N' \subseteq M$ is a nil-isomorphism as well.

    Since $M$ is coherent each $M/N_i$ is indeed nilpotent such that we can apply \autoref{t.MainLocalCase} to conclude that $M/ \bigcap N_i$ is nilpotent. Hence $N \defeq \bigcap N_i$ is the unique smallest nil-isomorphic subsheaf of $M$. It is clear that $N$ cannot have non-zero nilpotent quotients (since the kernel would be a strict subsheaf of $N$, nil-isomorphic to $M$, by \autoref{prop.nilp-abcat} (b)).

    By first replacing $M$ by $\overline{M}$ we can also achieve that condition (a) of the definition of minimality holds. As condition (a) passes to subsheaves, the smallest nil-isomorphic subsheaf of $\overline{M}$ is the sought after minimal $\gamma$--sheaf which is nil-isomorphic to $M$.
\end{proof}
\begin{rem}\label{r.finitelength}
    Essentially the same argument as in the proof of \autoref{t.MainLocalCase} shows the following: If $R$ is local and $M$ is a coherent $\gamma$--sheaf over $R$ with injective structural map, then any descending chain of $\gamma$--submodules of $M$ stabilizes. This was shown (with essentially the same argument) in \cite{Lyub} and implies immediately that $\gamma$--sheaves with injective structural map satisfy DCC.
\end{rem}

If one tries to reduce the general case of \autoref{t.MainLocal} (i.e. $R$ not local) to the local case just proven one encounters the problem of having to deal with the behavior of the infinite intersection $\bigcap N_i$ under localization. This is a source of troubles I do not know how to deal with directly. The solution to this is to take a detour and realize this intersection in a fashion such that each term functorially depends on $M$ and furthermore that this functorial construction commutes with localization. This is explained in the following section.

\subsection{\texorpdfstring{$D^{(1)}_X$}{D}--modules and Frobenius descent}

Let $D_X$ denote the sheaf of differential operators on $X$. This is a sheaf of rings on $X$ which locally, on each affine subvariety $\Spec R$ is described as follows.
\[
    D_R = \bigcup_{i=0}^\infty D^{(i)}_R
\]
where $D^{(i)}_R$ is the subset of $\End_{\mathbb{F}_q}(R)$ consisting of the operators which are linear over $R^{q^i}$, the subring of $(q^i)^{\text{th}}$ powers of elements of $R$. In particular $D^{(0)}_R \cong R$ and $D^{(1)}_R = \End_{R^{q}}(R)$. Clearly, $R$ itself becomes naturally a left $D^{(i)}_R$--module. Now denote by $R^{(1)}$ the $D^{(1)}_R$--$R$--bi-module which has this left $D^{(1)}_R$--module structure and the right $R$--module structure via Frobenius. i.e. for $r \in R^{(1)}$ and $x \in R$ we have $r \cdot x = rx^q$. With this notation we may view $D^{(1)}_R=\End^{\r}_R(R^{(1)})$ as the right $R$-linear endomorphisms of $R^{(1)}$. Thus we have
\[
    \Fr^*(\usc)=R^{(1)} \tensor_R \usc : R-\text{mod} \to D^{(1)}_R-\text{mod}
\]
which makes $\Fr^*$ into an equivalence of categories from $R$--modules to $D^{(1)}_R$--modules (because, since $\Fr$ is flat and $R$ is $F$--finite, $R^{(1)}$ is a locally free right $R$--module of finite rank). Its inverse functor is given by
\begin{equation}\label{eq:Fdesc}
    \Fr^{-1}(\usc)=\Hom^{\r}_R(R^{(1)},R) \tensor_{D^{(1)}_R} \usc: D^{(1)}_R-{\text{mod}} \to R-\text{mod}
\end{equation}
For details see \cite[Section 2.2]{AlvBliLyub.GeneratorsDmod}. I want to point out that these constructions commute with localization at arbitrary multiplicative sets. Let $S$ be a multiplicative set of $R$.\footnote{Since $S^{-1}R = (S^{[q]})^{-1}R$ we may assume that $S \subseteq R^q$. This implies that $S$ is in the center of $D^{(1)}_R$ such that localization in this non-commutative ring along $S$ is harmless. With this I mean that we may view the localization of the left $R$--module $D^{(1)}_R$ at $S^{-1}$ in fact as the localization of $D^{(1)}_R$ at the central multiplicative set $(S^{[q]})^{-1}$}. We have
\begin{equation}\label{eq:D-local}
\begin{split}
    S^{-1}D^{(1)}_R &= S^{-1}\End^{\r}_R(R^{(1)}) \\
                    &= \End^{\r}_{S^{-1}R}((S^{[q]})^{-1}R^{(1)}) =
    \End^{\r}_{S^{-1}R}((S^{-1}R)^{(1)}) \\
                    &= D^{(1)}_{S^{-1}R}
\end{split}
\end{equation}
Furthermore we have for an $D^{(1)}_R$--module $M$:
\[
\begin{split}
    S^{-1}(\sigma^{-1} M) &= S^{-1}(\Hom^{\r}_R(R^{(1)},R) \tensor_{D^{(1)}_R} M) \\
        &= S^{-1}\Hom^{\r}_R(R^{(1)},R) \tensor_{S^{-1}D^{(1)}_R} S^{-1}M \\
        &= \Hom^{\r}_{S^{-1}R}((S^{-1}R)^{(1)},S^{-1}R) \tensor_{D^{(1)}_{S^{-1}R}} S^{-1}M \\
        &= \sigma^{-1}(S^{-1}M)
\end{split}
\]
These observations are summarized in the following Proposition
\begin{prop}\label{t.commuteslocal}
Let $X$ be $F$--finite and regular. Let $U$ be an open subset (more generally, $U$ is locally given on $\operatorname{Spec} R$ as $\operatorname{Spec} S^{-1}R$ for some (sheaf of) multiplicative sets on $X$). Then
\[
    (D^{(1)}_X)|_U = D^{(1)}_U
\]
and for any sheaf of $D^{(1)}_X$--modules $M$ one has that
\[
    (\sigma^{-1}M)|_U = (\Hom^{\r}(\CO_X^{(1)},\CO_X) \tensor_{D^{(1)}_X} M)|_U \cong \Hom^{\r}(\CO_U^{(1)},\CO_U) \tensor_{D^{(1)}_U} M|_U = \sigma^{-1}(M|_U)
\]
as $\CO_U$--modules.
\end{prop}

\subsection{A criterion for minimality}

The Frobenius descent functor $\Fr^{-1}$ can be used to define an operation on $\gamma$--sheaves which assigns to a $\gamma$--sheaf $M$ its smallest $\gamma$--subsheaf $N$ with the property that $M/N$ has the trivial (=0) $\gamma$--structure. This is the opposite of what the functor $\Fr^*$ does: $\gamma: M \to \Fr^*M$ is a map of $\gamma$ sheaves such that $\Fr^*M/\gamma(M)$ has trivial $\gamma$--structure.

We define the functor $\FrG^{-1}$ from $\gamma$--sheaves to $\gamma$--sheaves as follows. Let $M \to[\gamma] \Fr^*M$ be a $\gamma$ sheaf. Then $\gamma(M)$ is an $\CO_X$--submodule of the $D^{(1)}_X$--module $\Fr^*M$. Denote by $D^{(1)}_X\gamma(M)$ the $D^{(1)}_X$--submodule of $\Fr^*M$ generated by $\gamma(M)$. To this inclusion of $D^{(1)}_X$--modules
\[
    D^{(1)}_X\gamma(M) \subseteq \Fr^*M
\]
we apply the Frobenius descent functor $\Fr^{-1}: D^{(1)}_X\text{--mod} \to \CO_X\text{--mod}$ defined above in \autoref{eq:Fdesc} and use that $\Fr^{-1} \circ \Fr^* = \id$ to define
\[
    \FrG^{-1}M \defeq \Fr^{-1}(D^{(1)}_X\gamma(M)) \subseteq \Fr^{-1}\Fr^*M = M
\]
In general one has $\FrG^{-1}(\Fr^*M) = \Fr^{-1}D^{(1)}_X\Fr^*(\gamma)(\Fr^*M) = \gamma(M)$ since $\Fr^*(\gamma)(\Fr^*M)$ already is a $D^{(1)}_X$--subsheaf of the $D^{(2)}_X$--module $\Fr^*(\Fr^*M)=\Fr^{2*}M$.

By construction $\FrG^{-1}M \subseteq M \to[\gamma] \gamma(M) \subseteq D^{(1)}_X \gamma(M) = \Fr^*\Fr^{-1}D^{(1)}_X \gamma(M) = \Fr^* \FrG^{-1} M$ such that $\FrG^{-1}M$ is a $\gamma$--subsheaf of $M$.

Furthermore, the quotient $M/\FrG^{-1}M$ has zero structural map. One makes the following observation
\begin{lem}\label{t.Mminimal}
    Let $M$ be a $\gamma$ sheaf. Then $\FrG^{-1}M$ is the smallest subsheaf $N$ of $M$ such that $\Fr^*N \supseteq \gamma(M)$.
\end{lem}
\begin{proof}
    Clearly $\Fr^{-1}M$ satisfies this condition. Let $N$ be as in the statement of the Lemma. Then $\Fr^*N$ is a $D^{(1)}_X$--subsheaf of $\Fr^*M$ containing $\gamma(M)$. Hence $D^{(1)}_X \gamma(M) \subseteq \Fr^*N$. Applying $\Fr^{-1}$ we see that $\sigma^{-1}M \subseteq N$.
\end{proof}
Therefore, the result of the lemma could serve as an alternative definition of $\sigma^{-1}_\gamma$ (one would have to show that the intersection of all such $N$ has again the property that $\gamma(M) \subseteq \Fr^*\bigcap N$ but this follows since $\Fr^*$ commutes with inverse limits). The following lemma is the key point in our reduction to the local case. It is an immediate consequence of \autoref{t.commuteslocal}. Nevertheless we include here a proof using only the characterization of \autoref{t.Mminimal}. Hence one may avoid the appearance of $D^{(1)}$--modules in this paper altogether but I believe it to be important to explain where the ideas for the arguments originated, hence $D^{(1)}$--modules are still there.
\begin{lem}\label{t:FrG-1_commutes_local}
    Let $M$ be a $\gamma$ sheaf and let $S \subseteq \CO_X$ be multiplicative set. Then $S^{-1}(\FrG^{-1}M) = \FrG^{-1}(S^{-1}M)$.
\end{lem}
\begin{proof}
    This follows from \autoref{t.commuteslocal}. However, this can also be proven using only the characterization in \autoref{t.Mminimal}: By this we have
    \begin{equation}\label{eq.loc1}
        \Fr^*(S^{-1}(\FrG^{-1}M)) = S^{-1}(\Fr^*(\FrG^{-1})) \supseteq S^{-1}\gamma(M) = \gamma(S^{-1}M)
    \end{equation}
    which implies that $\FrG^{-1}(S^{-1}M) \subseteq S^{-1}(\FrG^{-1}M)$ because $\FrG^{-1}(S^{-1}M)$ is smallest (by \autoref{t.Mminimal}) with respect to the inclusion shown in the displayed equation \autoref{eq.loc1}. On the other hand one has the chain of inclusions
    \[
    \begin{split}
        \Fr^*(M\cap S^{-1}\FrG^{-1}(M)) &=\Fr^*M \cap \Fr^*\FrG^{-1}(S^{-1}M) \\
                                        &\supseteq \Fr^*M \cap \gamma(S^{-1}M) \supseteq \gamma(M)
    \end{split}
    \]
    and hence \autoref{t.Mminimal} applied to $M$ yields
    \[
        \FrG^{-1}M \subseteq M \cap S^{-1}\FrG^{-1}(M)\, .
    \]
    Therefore $S^{-1}\FrG^{-1}M \subseteq S^{-1}M \cap S^{-1}(\FrG^{-1}S^{-1}M) = \FrG^{-1}S^{-1}M$ which finishes the argument.
\end{proof}

\begin{prop}\label{t.Minimal=Stable-Open}
    Let $M$ be a $\gamma$-sheaf. Then $\FrG^{-1}M = M$ if and only if $M$ has no proper nilpotent quotients (i.e. satisfies condition (b) of the definition of minimality)

    If $M$ is coherent. The condition on $x \in X$ that the inclusion $\FrG^{-1}(M_x) \subseteq M_x$ is equality is an open condition on $X$.
\end{prop}
\begin{proof}
    One direction is clear since $M/\FrG^{-1}M$ is a nilpotent quotient of $M$.
    We use the characterization in \autoref{prop.char-quot}. For this let $N \subseteq M$ be such that $\gamma(M) \subseteq \Fr^*N$. $\FrG^{-1}M$ was the smallest subsheaf with this property, hence $\FrG^{-1}M \subseteq N \subseteq M$. Since $M = \FrG^{-1}M$ by assumption it follows that $N=M$. Hence, by \autoref{prop.char-quot}, $M$ does not have non-trivial nilpotent quotients.

    By \autoref{t:FrG-1_commutes_local} $\FrG^{-1}$ commutes with localization which means that $\FrG^{-1}(M_x)=(\FrG^{-1}M)_x$. Hence the second statement follows simply since both $M$ and $\FrG^{-1}M$ are coherent (and equality of two coherent modules via a given map is an open condition).
\end{proof}
\begin{lem}\label{t.end-exact}
    The assignment $M \mapsto \FrG^{-1} M$ is an end-exact functor on $\gamma$--sheaves.
\end{lem}
\begin{proof}
Formation of the image of the functorial map $\id \to[\gamma] \Fr^*$ of exact functors is end-exact (see for example \cite[2.17 Appendix 1]{Katz.RigidLocal}). If $M$ is a $D^{(1)}_X$--module and $A \subseteq B$ are $\CO_X$--submodules of $M$ then $D^{(1)}_XA \subseteq D^{(1)}_XB$. If $M \onto N$ is a surjection of $D^{(1)}$--modules with induces a surjection on $\CO_X$--submodules $A \onto B$ then, clearly, $D^{(1)}_X A$ surjects onto $D^{(1)}_X B$. Now one concludes by observing that $\Fr^{-1}$ is an exact functor.
\end{proof}

\begin{lem}\label{t.iteratedFrobInv}
    Let $N \subseteq M$ be an inclusion of $\gamma$--sheaves such that $\sigma^{n*} N \supseteq \gamma^n(M)$ (i.e. the quotient is nilpotent of order $\leq n$). Then $\sigma^{(n-1)*}(N \cap \FrG^{-1}M) \supseteq \gamma^{n-1}(\FrG^{-1}M)$.
\end{lem}
\begin{proof}
    Consider the $\gamma$--subsheaf $M'=(\gamma^{n-1})^{-1}(\Fr^{(n-1)*}N)$ of $M$. One has
    \[
        \Fr^*M'=(\Fr^*\gamma^{n-1})^{-1}(\Fr^{n*}N) \supseteq \gamma(M)
    \]
    by the assumption that $\gamma^n(M) \subseteq \sigma^{n*}N$. Since $\FrG^{-1}M$ is minimal with respect to this property we have $\FrG^{-1}M \subseteq (\gamma^{n-1})^{-1}(\Fr^{(n-1)*}N)$. Applying $\gamma^{n-1}$ we conclude that $\gamma^{n-1}(\FrG^{-1}M) \subseteq \Fr^{(n-1)*}N$. Since $\FrG^{-1}M$ is a $\gamma$--sheaf we have $\gamma(\FrG^{-1}M) \subseteq \Fr^{(n-1)*}(\FrG^{-1}M)$ such that the claim follows.
\end{proof}

\subsection{Existence of minimal \texorpdfstring{$\gamma$}{Gamma}--sheaves}\label{s.max_nil_quotient}

For a given $\gamma$--sheaf $M$ we can iterate the functor $\FrG^{-1}$ to obtain a decreasing sequence of $\gamma$--subsheaves
\[
    \ldots \subseteq M_{3} \subseteq M_{2} \subseteq M_{1} \subseteq M (\to[\gamma] \Fr^*M \to \ldots)
\]
where $M_{i} = \FrG^{-1}M_{i-1}$. Note that each inclusion $M_{i} \subseteq M_{i-1}$ is a nil-isomorphism.
\begin{prop}\label{prop.const-miniml}
    Let $M$ be a coherent $\gamma$--sheaf. Then the following conditions are equivalent.
    \begin{enumerate}
        \item $M$ has a nil-isomorphic $\gamma$--subsheaf $\underline{M}$ which does not have non-trivial nilpotent quotients (i.e. $\underline{M}$ satisfies condition (b) in the definition of minimal $\gamma$--sheaf).
        \item $M$ has a \emph{unique} smallest nil-isomorphic subsheaf (equiv. $M$ has a (unique) maximal nilpotent quotient).
        \item For some $n \geq 0$, $M_{n} = M_{n+1}$.
        \item There is $n \geq 0$ such that for all $m \geq n$, $M_{m} = M_{m+1}$.
    \end{enumerate}
\end{prop}
\begin{proof}
  (a) $\Rightarrow$ (b): Let $\underline{M} \subseteq M$ be the nil-isomorphic subsheaf of part (a) and let $N \subseteq M$ be another nil-isomorphic subsheaf of $M$. By \autoref{t.lemma.basic} it follows that $\underline{M} \cap N$ is also nil-isomorphic to $M$. In particular $\underline{M}/(\underline{M} \cap N)$ is a nilpotent quotient of $\underline{M}$ and hence must be trivial. Thus $N \subseteq \underline{M}$ which shows that $\underline{M}$ is the smallest nil-isomorphic subsheaf of $M$.

  (b) $\Rightarrow$ (c): Let $N$ be this smallest subsheaf as in (b). Since each $M_{i}$ is nil-isomorphic to $M$, it follows that $N \subseteq M_i$ for all $i$. Let $n$ be the order of nilpotency of the quotient $M/N$, i.e. $\gamma^n(M) \subseteq \Fr^{n*}N$. Repeated application ($n$ times) of \autoref{t.iteratedFrobInv} yields that $M_{n} \subseteq N$. Hence we get $N \subseteq M_{n+1} \subseteq M_{n} \subseteq N$ which implies that $M_{n+1} = M_{n}$.

  (c) $\Rightarrow$ (d) is clear.

  (d) $\Rightarrow$ (a) is clear by \autoref{t.Minimal=Stable-Open}.
\end{proof}

This characterization enables us to show the existence of minimal $\gamma$-sheaves by reducing to the local case which we proved above.

\begin{thm}\label{thm.main}
    Let $M$ be a coherent $\gamma$--sheaf. There is a unique subsheaf $\underline{M}$ of $M$ which does not have non-trivial nilpotent quotients.
\end{thm}
\begin{proof}
    By \autoref{prop.const-miniml} it is enough to show that the sequence $M_{i}$ is eventually constant. Let $U_{i}$ be the subset of $X$ consisting of all $x \in X$ on which $(M_{i})_x= (M_{i+1})_x (= (\FrG^{-1}M_{i})_x)$. By \autoref{t.Minimal=Stable-Open} $U_{i}$ is an open subset of $X$ (in this step I use the key observation \autoref{t.commuteslocal}) an that $(M_{i})|_{U_{i}}=(M_{i+1})|_{U_{i}}$. By the functorial construction of the $M_{i}$'s the equalilty $M_{i}=M_{i+1}$ for one $i$ implies equality for all bigger $i$. It follows that the sets $U_{i}$ form an increasing sequence of open subsets of $X$ whose union is $X$ itself by \autoref{t.MainLocal} and \autoref{prop.const-miniml}. Since $X$ is noetherian, $X=U_{i}$ for some $i$. Hence $M_{i}=M_{i+1}$ such that the claim follows by \autoref{prop.const-miniml}.
\end{proof}

\begin{thm}\label{t.MinFunct}
    Let $M$ be a coherent $\gamma$--sheaf. Then there is a functorial way to assign to $M$ a \emph{minimal} $\gamma$--sheaf $M_{\min}$ in the nil-isomorphism class of $M$.
\end{thm}
\begin{proof}
    We may first replace $M$ by the nil-isomorphic quotient $\overline{M}$ which satisfies condition (a) of \autoref{d.minimal}. Then replace $\overline{M}$ by its minimal nil-isomorphic submodule $\underline{(\overline{M})}$ which also satisfies condition (b) of \autoref{d.minimal} (and condition (a) because (a) is passed to submodules). Thus the assignment $M \mapsto M_\min \defeq \underline{(\overline{M})}$ is a functor since it is a composition of the functors $M \mapsto \overline{M}$ and $M \mapsto \underline{M}$.
\end{proof}

\begin{prop}\label{t.nil-iso}
    If $\phi \colon M \to N$ is a nil-isomorphism, then $\phi_\min \colon M_\min \to N_\min$ is an isomorphism.
\end{prop}
\begin{proof}
    Clearly, $\phi_\min$ is a nil-isomorphism. Since $\Ker \phi_\min$ is a nilpotent subsheaf of $M_\min$, we have by \autoref{d.minimal} (a) that $\Ker \phi_\min = 0$. Since $\Coker \phi_\min$ is a nilpotent quotient of $N_\min$ it must be zero by \autoref{d.minimal} (b).
\end{proof}

\begin{cor}
    Let $\CM$ be a finitely generated unit $\CO_{X}[\Fr]$--module. Then $M$ has a unique minimal root in the sense of \cite{Lyub}.
\end{cor}
\begin{proof}
    Let $M$ be any root of $\CM$, i.e. $M$ is a coherent $\gamma$--sheaf such that $\gamma_M$ is injective and $\Gen M \cong \CM$. Then $M_\min = \underline{M}$ is a minimal nil-isomorphic $\gamma$--subsheaf of $M$ by \autoref{t.MinFunct}. By \autoref{t.min=minroot} it follows that $M_\min$ is the sought after minimal root of $\CM$.
\end{proof}
Note that the only assumption needed in this result is that $X$ is $F$--finite and regular. In particular it does not rely on the finite--length result \cite{Lyub} Theorem 3.2 which assumes that $R$ is of finite type over a regular local ring (however if does not assume $F$--finiteness).

\begin{thm}
    Let $X$ be regular and $F$--finite. Then the functor
    \[
        \Gen \colon \MinG(X) \to \text{finitely generated unit $\CO_X[\Fr]$--modules}
    \]
    is an equivalence of categories.
\end{thm}
\begin{proof}
    The preceding corollary shows that $\Gen$ is essentially surjective. The induced map on $\Hom$ sets is injective since a map of minimal $\gamma$--sheaves $f$ is zero if and only if its image is nilpotent (since minimal $\gamma$--sheaves do not have nilpotent sub-modules) which is the condition that $\Gen(f) = 0$. It is surjective since any map between $g:\Gen(M) \to \Gen(N)$ is obtained from a map of $\gamma$--sheaves $M \to \Fr^{e*}N$ for some $e \gg 0$. But this induces a map $M=M_\min \to (\Fr^{e*}N)_\min=N_\min=N$.
\end{proof}

\section{Applications and Examples}\label{s.final}
In this section we discuss some further examples and applications of the results on minimal $\gamma$--sheaves we obtained so far.

\subsection{\texorpdfstring{$\gamma$}{Gamma}--crystals}
The purpose of this section is to quickly explain the relationship of minimal $\gamma$--sheaves to $\gamma$-crystals which were introduced in \cite{BliBoe.CartierCrys}. The category of $\gamma$--crystals is obtained by inverting nil-isomorphisms in $\CohG(X)$. In \cite{BliBoe.CartierCrys} it is shown that the resulting category is abelian. One has a natural functor
\[
    \CohG(X) \onto \CrysG(X)
\]
whose fibers we may think of consisting of nil-isomorphism classes of $M$. Note that the objects of $\CrysG(X)$ are the same as in $\CohG(X)$, however a morphism between $\gamma$--crystals $M \to N$ is represented by a left-fraction, i.e. a diagram of $\gamma$--sheaves $M \Leftarrow M' \rightarrow M$ where the arrow $\Leftarrow$ is a nil-isomorphism.

On the other hand we just constructed the subcategory of minimal $\gamma$--sheaves $\MinG(X) \subseteq \CohG(X)$ and showed that there is a functorial splitting $M \mapsto M_\min$ of this inclusion. An immediate consequence of \autoref{t.nil-iso} is that if $M$ and $N$ are in the same nil-isomorphism class, then $M_\min \cong N_\min$. The verification of this may be reduced to considering the situation
\[
    M \Leftarrow M' \Rightarrow N
\]
with both maps nil-isomorphisms in which case \autoref{t.nil-iso} shows that $M_\min \cong M'_\min \cong N_\min$. One has the following Proposition.
\begin{prop}
    Let $X$ be regular and $F$--finite. Then the composition
    \[
    \MinG(X) \into \CohG(X) \onto \CrysG(X)
    \]
    is an equivalence of categories whose inverse is given by sending a $\gamma$--crystal represented by the $\gamma$--sheaf $M$ to the minimal $\gamma$--sheaf $M_\min$.
\end{prop}
\begin{proof}
    The existence of $M_\min$ shows that $\MinG(X) \to \CrysG(X)$ is essentially surjective. It remains to show that $\Hom_{\MinG}(M,N) \cong \Hom_{\CrysG}(M,N)$. A map $\phi\colon M \to N$ of minimal $\gamma$--sheaves is zero in $\CrysG$ if and only if $\Image \phi$ is nilpotent. But $\Image \phi$ is a subsheaf of the minimal $\gamma$--sheaf $N$, which by \autoref{d.minimal} (a) has no nontrivial nilpotent subsheaves. Hence $\Image \phi = 0$ and therefore $\phi = 0$. This shows that the map on $\Hom$ sets is injective. The surjectivity follows again by functoriality of $M \mapsto M_{\min}$.
\end{proof}

\begin{cor}
    Let $X$ be regular and $F$--finite. The category of minimal $\gamma$--sheaves $\MinG(X)$ is an abelian category. If $\phi \colon M \to N$ is a morphism then $\Ker_\min \phi = (\Ker \phi)_\min = \underline{\Ker \phi}$ and $\Coker_\min \phi = (\Coker \phi)_\min = \overline{\Coker \phi}$.
\end{cor}
\begin{proof}
    Since $\MinG(X)$ is equivalent to $\CrysG(X)$ and since the latter is abelian, so is $\MinG(X)$. This implies also the statement about $\Ker$ and $\Coker$.
\end{proof}

\subsection{The parameter test module}
We give an application to the theory of tight closure. In \cite{Bli.int} Proposition 4.5 it was shown that the parameter test module $\tau_{\omega_A}$ is the unique minimal root of the intersection homology unit module $\CL \subseteq H^{n-d}_I(R)$ if $A=R/I$ is the quotient of the regular local ring $R$ (where $\dim R=n$ and $\dim A=d$). Locally, the parameter test module $\tau_{\omega_A}$ is defined as the Matlis dual of $H^d_m(A)/0^*_{H^d_m(A)}$ where $0^*_{H^d_m(A)}$ is the tight closure of zero in $H^d_m(A)$.
The fact that we are now able to construct minimal $\gamma$--sheaves globally allows us to give a global candidate for the parameter test module.
\begin{prop}\label{t.GlobalParTest}
    Let $A = R/I$ where $R$ is regular and $F$--finite. Then there is a submodule $L \subseteq \omega_A= \Ext^{n-d}(R/I,R)$ such that for each $x \in \Spec A$ we have $L_x \cong \tau_{\omega_x}$.
\end{prop}
\begin{proof}
    Let $\CL \subseteq H^{n-d}_I(R)$ be the unique smallest submodule of $H^{n-d}_I(R)$ which agrees with $H^{n-d}_I(R)$ on all smooth points of $\Spec A$. $\CL$ exists by \cite{Bli.int} Theorem~4.1. Let $L$ be a minimal generator of $\CL$, i.e. a coherent minimal $\gamma$--sheaf such that $\Gen L = \CL$ which exists due to \autoref{thm.main}. Because of \autoref{t.commuteslocal} it follows that $L_x$ is also a minimal $\gamma$--sheaf and $\Gen L_x \cong \CL_x$. But from \cite{Bli.int} Proposition~4.5 we know that the unique minimal root of $\CL_x$ is $\tau_{\omega_{A_x}}$, the parameter test module of $A_x$. It follows that $L_x \cong \tau_{\omega_{A_x}}$ by uniqueness. To see that $L \subseteq \Ext^{n-d}(R/I,R)$ we just observe that $\Ext^{n-d}(R/I,R)$ with the map induced by $R/I^{[q]} \to R/I$ is a $\gamma$--sheaf which generates $H^{n-d}_I(R)$. Hence by minimality of $L$ we have the desired inclusion.
\end{proof}

\subsection{Test ideals and minimal $\gamma$--sheaves}
We consider now the simplest example of a $\gamma$--sheaf, namely that of a free rank one $R$--module $M (\cong R)$. That means that via the identification $R \cong \Fr^*R$ the structural map
\[
    \gamma: M \cong R \to[f \cdot] R \cong \Fr^*R \cong \Fr^*M
\]
is given by multiplication with an element $f \in R$. It follows that $\gamma^e$ is given by multiplication by $f^{1+q+\ldots+q^{e-1}}$ under the identification of $\Fr^{e*}R \cong R$

We will show that the minimal $\gamma$--subsheaf of the just described $\gamma$-sheaf $M$ can be expressed in terms of generalized test ideals. We recall from \cite{BliMusSmi.hyp} Lemma \ref*{lem0} that the test ideal of a principal ideal $(f)$ of exponent $\alpha= \frac{m}{q^e}$ is given by
\[
    \tau(f^\alpha) = \text{smallest ideal $J$ such that $f^m \in J^{[q^e]}$}
\]
by Lemma \ref*{lem1} of op. cit. $\tau(f^\alpha)$ can also be characterized as $\Fr^{-e}$ of the $D^{(e)}$--module generated by $f^m$. We set as a shorthand $J_e = \tau(f^{(1+q+q^2+\ldots+q^{e-1})/q^e})$ and repeat the definition:
\begin{align*}
    J_e = \text{smallest ideal $J$ of $R$ such that $f^{1+q+q^2+\ldots+q^{e-1}} \in J^{[q^e]}$}
\end{align*}
and further recall from \autoref{s.max_nil_quotient} that
\begin{align*}
    M_e &= \text{smallest ideal $I$ of $R$ such that $f\cdot M_{e-1} \subseteq I^{[q]}$}
\end{align*}
with $M_0=M$.
\begin{lem}
    For all $e \geq 0$ one has $J_e = M_e$.
\end{lem}
\begin{proof}
    The equality is true for $e=1$ by definition. We first show the inclusion $J_e \subseteq M_e$ by induction on $e$.
    \[
    \begin{split}
        M_e^{[p^e]} &\supseteq (f\cdot M_{e-1})^{[q^{e-1}]} = (f^{q^{e-1}}M_{e-1}^{[q^{e-1}]}) \\
                    &= (f^{q^{e-1}}J_{e-1}^{[q^{e-1}]}) \supseteq f^{q^{e-1}} \cdot f^{1+q+q^2+\ldots+q^{e-2}} \\
                    &=f^{1+q+q^2+\ldots+q^{e-1}}
    \end{split}
    \]
    since $J_e$ is minimal with respect to this inclusion we have $J_e \subseteq M_e$.

    Now we show for all $e \geq 1$ that $f \cdot J_{e-1} \subseteq J_e^{[q]}$. The definition of $J_e$ implies that
    \[
        f^{1+q+\ldots+q^{e-2}} \in (J^{[q^e]} \colon f^{q^{e-1}}) = (J^{[q]} \colon f)^{[q^{e-1}]}
    \]
    which implies that $J_{e-1} \subseteq (J^{[q]} \colon f)$ by minimality of $J_{e-1}$. Hence $f\cdot J_{e-1} \subseteq J^{[q]}$. Now, we can show the inclusion $M_e \subseteq J_e$ by observing that by induction one has
    \[
        J_e^{[q]} \supseteq f \cdot J_{e-1} \supseteq f \cdot M_{e-1}\, .
    \]
    which implies by minimality of $M_e$ that $M_e \subseteq J_e$.
\end{proof}
This shows that the minimal $\gamma$--sheaf $M_\min$, which is equal to $M_e$ for $e \gg 0$ by \autoref{prop.const-miniml}, is just the test ideal $\tau(f^{(1+q+q^2+\ldots+q^{e-1})/q^e})$ for $e \gg 0$. As a consequence we have:
\begin{prop}\label{t.MinimalTestIdealChar}
    Let $M$ be the $\gamma$--sheaf given by $R \to[f\cdot] R \cong \Fr^*R$. Then $M_\min = \tau(f^{(1+q+q^2+\ldots+q^{e-1})/q^e})$ for $q \gg 0$. In particular, $M_\min \supseteq \tau(f^{\frac{1}{q-1}})$ and the $F$-pure-threshold of $f$ is $\geq \frac{1}{q-1}$ if and only if $M$ is minimal.
\end{prop}
\begin{proof}
    For $e \gg 0$ the increasing sequence of rational numbers $(1+q+q^2+\ldots+q^{e-1})/q^e$ approaches $\frac{1}{q-1}$. Hence $M_e = \tau(f^{(1+q+q^2+\ldots+q^{e-1})/q^e}) \supseteq \tau(f^{\frac{1}{q-1}})$ for all $e$. If $M$ is minimal, then all $M_e$ are equal hence the multiplier ideals $\tau(f^\alpha)$ must be equal to $R$ for all $\alpha \in [0,\frac{1}{q-1})$. In particular, the $F$-pure-threshold of $f$ is $\geq \frac{1}{q-1}$. Conversely, if the $F$--pure threshold is less than $\frac{1}{q-1}$, then for some $e$ we must have that $\tau(f^{(1+q+q^2+\ldots+q^{e-1})/q^e}) \neq \tau(f^{(1+q+q^2+\ldots+q^{e})/q^{e+1}})$ such that $M_e \neq M_{e+1}$ which implies that $M \neq M_1$ such that $M$ is not minimal.
\end{proof}

\begin{rem}
    This shows also, after replacing $f$ by $f^r$, that $\frac{r}{q-1}$ is not an accumulation point of $F$--thresholds of $f$ for any $f$ in an $F$--finite regular ring. In \cite{BliMusSmi.hyp} this was shown for $R$ essentially of finite type over a local ring since our argument there depended on \cite{Lyub} Theorem 4.2. Even though $D$--modules appear in the present article, they only do so by habit of the author, as remarked before, they can easily be avoided. 
\end{rem}
\begin{rem}
    Of course, for $r=q-1$ this recovers (and slightly generalizes) the main result in \cite{AlvBliLyub.GeneratorsDmod}.
\end{rem}
\begin{rem}
    I expect that this descriptions of minimal roots can be extended to a more general setting using the modifications of generalized test ideals to modules as introduced in the preprint \cite{TakagiTakahashi.DmodFFRT}.
\end{rem}

\bibliographystyle{amsalpha}
\bibliography{./../MyBibliography}
\end{document}